\documentclass{article}
\usepackage{amsthm}
\usepackage{amssymb}
\usepackage{amsfonts}
\usepackage{amsmath}

\newtheorem{theorem}{Theorem}[section]
\newtheorem{lemma}[theorem]{Lemma}
\newtheorem{proposition}[theorem]{Proposition}
\newtheorem{corollary}[theorem]{Corollary}

\theoremstyle{definition}
\newtheorem{example}[theorem]{Example}

\newtheorem{remark}[theorem]{Remark}
\input{tcilatex}
\begin{document}

\begin{center}
\bigskip {\huge On} {\huge the Dec group of finite abelian Galois extensions
over global fields}.

\bigskip Jean B. Nganou\footnote{%
The results of this paper will be included in the Author PhD's dissertation.
The author is very thankful to his advisor Dr. Morandi for his guidance. The
author would also like to thank his Mathematical grandfather Dr. A.
Wadsworth for his suggestions on improving the results of this paper.}

New Mexico State University

Department of Mathematical Sciences

P.O. Box 30001, Department 3MB

Las Cruces, New Mexico 88003-8001

Email: jganou@nmsu.edu

\bigskip

\bigskip

\bigskip

\bigskip
\end{center}

\textbf{Abstract:} If $K/F$ is a finite abelian Galois extension of global
fields whose Galois group has exponent $t$, we prove that there exists a
short exact sequence $0\longrightarrow \func{Dec}(K/F)\longrightarrow \func{%
Br}_{t}(K/F)\longrightarrow \underset{q\in \mathcal{P}}{\oplus }r_{q}%
\mathbb{Z}
/%
\mathbb{Z}
\longrightarrow 0$ where $r_{q}\in 
\mathbb{Q}
$ and $\mathcal{P}$ is a finite set of primes of $F$ that is empty if $t$ is
square free. In particular, we obtain that if $t$ is square free, then $%
\func{Dec}(K/F)=\func{Br}_{t}(K/F)$ which we use to show that prime exponent
division algebras over Henselian valued fields with global residue fields
are isomorphic to a tensor product of cyclic algebras. Finally, we construct
a counterexample to the result for higher exponent division algebras.

\bigskip

\textbf{Introduction:} For any field $F$, $\func{Br}(F)$ denotes the Brauer
group of F and for any positive integer $n,$\ $\func{Br}_{n}(F)$ denotes the
subgroup of $\func{Br}(F)$ whose elements are the the $F$-central division
algebras whose exponent divides $n$. If $K/F$ is any field extension of $F$, 
$\func{Br}(K/F)$ denotes the kernel of the restriction map $\func{Br}%
(F)\rightarrow \func{Br}(K),$ and $\func{Br}_{n}(K/F)=\func{Br}(K/F)\cap 
\func{Br}_{n}(F).$ In addition the Dec group of $K/F$ denoted by $\func{Dec}%
(K/F)$ is the subgroup of $\func{Br}(K/F)$ generated by $\cup \func{Br}(L/F)$
\ where the union ranges over the cyclic subextensions $L/F$ of $K/F.$

If $F$ is a field containing a primitive $n^{th}$ root of unity $\omega $
and $a,b\in F^{\ast }$. $A_{\omega }(a,b)$ will denote the symbol algebra
over $F$ with generators $i,j$ satisfying $i^{n}=a,$ $j^{n}=b,$ $ij=\omega
ji.$ Furthermore, when there is no risk of confusion, $A_{-1}(a,b)$ will be
denoted by $(a,b).$

In addition assume that $F$ contains a primitive $n^{th}$ root of unity with 
$n=\exp (D),$ then it is also known that if $l(D)$ denotes the minimum
number of symbols required to represent the Brauer class of $D$ (recall that
such an integer exists by the Merkurjev-Suslin Theorem), then $n(D)\leq
l(D). $ Also, numerous examples of division algebras of prime exponent have
been constructed by several authors with $n(D)<l(D)$, [see e.g., Ja, Sa, Ti$%
_{3}$]$.$Note that indecomposable division algebras satisfy the inequality.
At the heart of many of these examples, lies the obstruction between the
relative Brauer group and the Dec group. The aim of this paper is to study
this obstruction for global fields. If $K/F$ is a finite abelian Galois
extension of global fields whose Galois group has exponent $t$, we define
what is called a bad prime of $F,$ then prove that there exists a short
exact sequence $0\longrightarrow \func{Dec}(K/F)\longrightarrow \func{Br}%
_{t}(K/F)\longrightarrow \underset{\text{bad }q^{\prime }s}{\oplus }\frac{%
c_{q}}{d_{q}}%
\mathbb{Z}
/%
\mathbb{Z}
\longrightarrow 0$ where $c_{q}$ and $d_{q}$ are positive integers with $%
c_{q}$\TEXTsymbol{\vert}$d_{q}.$ Observing that $F$ has no bad primes when $%
t $ is square free, we deduce that\ $\func{Dec}(K/F)=\func{Br}_{t}(K/F)$
whenever $t$ is square free. Note that the special case ($t=2$)\ of this
result was proved using quadratic form theory [ELTW, Corollary 3.18]. We use
this result to show that if $L$ is a Henselian valued field containing a
primitive $p^{th}$ root of unity whose residue field is a global field, then
any exponent $p$ division algebra over $L$ is isomorphic to a tensor product
of degree $p$ symbol algebras. The last part of this paper is devoted to
constructing a counterexample that shows that non square free exponent
division algebras over the class of fields studied may not be isomorphic to
a product of symbol algebras. In fact, we construct an example of an
exponent 4 division algebra over $%
\mathbb{Q}
(i)((x))((y))$ that is not isomorphic to a product of symbol algebras of
exponent dividing 4.

If $F$ is a global field, $P$ a prime of $F$ and $D$ is an $F$-algebra $%
r_{F_{P}/F}(D)\doteqdot D\otimes _{F}F_{P}$ will be denoted by $D_{P}.$

\section{The Dec group of finite abelian Galois extensions over global fields%
}

We start this section by the following easy lemma.

\begin{lemma}
\label{decin}Let $K/F$ a finite abelian Galois extension with galois group $%
G $. If $t$ denotes the exponent of $G$, then $\func{Dec}(K/F)\subseteq 
\func{Br}_{t}(K/F)$
\end{lemma}

\begin{proof}
Since $\func{Dec}(K/F)$ is generated by the relative Brauer groups $\func{Br}%
(L/F)$ where $L/F$ runs over cyclic subextensions of $K/F,$ it is enough to
show that $\func{Br}(L/F)\subseteq \func{Br}_{t}(K/F)$ for every $L/F$
cyclic subextensions of $K/F.$ Note that for any such cyclic extension $L/F,$
as $\mathcal{G(}L/F\mathcal{)}$ is a quotient group of $G,$ then $[L:F]=|%
\mathcal{G(}L/F\mathcal{)}|=\exp (\mathcal{G(}L/F\mathcal{)})$ divides $\exp
(G)=t.$ Therefore if $D\in \func{Br}(L/F),$ then $\exp (D)$ divides $\func{%
Ind}(D)$ which divides $[L:F]$, hence $\exp (D)$ divides $t.$ Thus for every
cyclic subextension $L/F$ of $K/F,$ $\func{Br}(L/F)\subseteq \func{Br}%
_{t}(K/F)$, therefore $\func{Dec}(K/F)\subseteq \func{Br}_{t}(K/F).$
\end{proof}

For local fields, the inclusion above is an equality as we now show.

\begin{proposition}
\label{decloc}Let $F=%
\mathbb{R}
,%
\mathbb{C}
,$ or a local field, $K/F$ a finite abelian Galois extension whose Galois
group has exponent $s.$ Then $\func{Dec}(K/F)=\func{Br}_{s}(K/F).$
\end{proposition}

\begin{proof}
We only need to prove that $\func{Br}_{s}(K/F)\subseteq \func{Dec}(K/F).$ If 
$F=%
\mathbb{R}
$ or $%
\mathbb{C}
,$ then $K=%
\mathbb{R}
$ or $%
\mathbb{C}
,$ and in either case the inclusion is obvious. Suppose that $F$ is a local
field and let $D\in $ $\func{Br}_{s}(K/F).$ Note that there exists a cyclic
subextension $L/F$ of $K/F$ of degree $s$ (for example, if $\mathcal{G}(K/F)=%
\mathbb{Z}
_{s}\times 
\mathbb{Z}
_{t_{2}}\times ...\times 
\mathbb{Z}
_{t_{m}},$ take $L=\mathcal{F}(%
\mathbb{Z}
_{t_{2}}\times ...\times 
\mathbb{Z}
_{t_{m}})$). Therefore $\func{Ind}(D)=\exp (D)$ divides $t=[L:F],$ hence by
[Re, Corollary 31.10], $D\in \func{Br}(L/F)\subseteq \func{Dec}(K/F),$ thus $%
D\in \func{Dec}(K/F)$ as needed.
\end{proof}

\begin{lemma}
\label{arch}Let $F$ a global field, $P$ a non-archimedian prime of $F$, $%
0\leq r\leq n,$ and let $L/F$ be a cyclic extension such that $%
[L_{P}:F_{P}]=n.$ If $\sigma $ denotes a generator for $\mathcal{G}%
(L_{P}/F_{P}),$ then there exists $b\in F_{P}$ so that $\func{Inv}%
_{P}((L_{P}/F_{P},\sigma ,b))=\frac{r}{n}+%
\mathbb{Z}
.$
\end{lemma}

\begin{proof}
Recall (see, e.g., [Re,]) that $\func{Inv}_{P}:\func{Br}(F_{P})\rightarrow 
\mathbb{Q}
/%
\mathbb{Z}
$ is an isomorphism, therefore there exists $A\in \func{Br}(F)$ so that $%
\func{Inv}_{P}(A)=\frac{r}{n}+%
\mathbb{Z}
.$ But this implies that $\func{Inv}_{P}(A^{n})=%
\mathbb{Z}
,$ therefore $A^{n}\sim 1,$ hence $\func{Ind}(A)=\exp (A)$ divides $%
n=[L_{P}:F_{P}].$ Thus $A\in \func{Br}(L_{P}/F_{P}),$ and by [Dr, Theorem 1
pp.71], there exists $b\in F_{P}$ so that $A\sim (L_{P}/F_{P},\sigma ,b),$
therefore $\func{Inv}_{P}((L_{P}/F_{P},\sigma ,b))=\func{Inv}_{P}(A)=\frac{r%
}{n}+%
\mathbb{Z}
.$
\end{proof}

We set up the notations for the remaining of this section. $F$ is a global
field and $K/F$ is a finite abelian Galois extension with Galois group $G$
of exponent $t$. For any prime $q$ of $F$, $G_{q}^{K/F}$ denotes the local
Galois group $\mathcal{G}(K_{q}/F_{q})$ which we will simply denote by $%
G_{q} $ when there is no risk of consfusion. Also denote by $c_{q}$ the
exponent of $G_{q}$ and $d_{q}:=\gcd (t,|G_{q}|).$ Note that as $G_{q}$ is
isomorphic to a subgroup of $G$, then $\exp (G_{q})$ divides $\exp (G),$ but
also $\exp (G_{q})$ divides $|G_{q}|,$ hence $c_{q}$ divides $d_{q}.$ We
call a prime $q $ bad if $c_{q}<d_{q}.$

\begin{remark}
\label{finbad}If $q$ is a bad prime of $F,$ then $G_{q}$ is not cyclic,
therefore $q$ is ramified because if $q$ is unramified, then $G_{q}\cong 
\mathcal{G}(\overline{K}/\overline{F})$ which is cyclic. But since there are
only finitely many ramified primes, it follows that there are only finitely
many bad primes.
\end{remark}

We need the following group theoretic results in Proposition \ref{tchebo}
below.

\begin{proposition}
\label{grpth1}Let $G$ a finite abelian group of exponent $t$. Then there
exist an integer $n\geq 1,$ a family $\{H_{i}\}_{i=1}^{n}$ of cyclic
subgroups of $G$ of order $t,$ and a subgroup $D$ of $G$ such that $G/D$ is
cyclic of order $t$ satisfying $H_{1}\cdot H_{2}\cdots H_{n}=G$ and $%
H_{i}\cdot D=G$ for all $i.$
\end{proposition}

\begin{proof}
Write $G=%
\mathbb{Z}
_{t}\times 
\mathbb{Z}
_{t_{2}}\times ...\times 
\mathbb{Z}
_{t_{n}}$ with $t_{i}|t$ for all $i$ and set $D=%
\mathbb{Z}
_{t_{2}}\times ...\times 
\mathbb{Z}
_{t_{n}},$ then $G/D$ is cyclic of order $t.$ Let $\sigma ,\sigma
_{2},...,\sigma _{n}$ be the generators of $%
\mathbb{Z}
_{t},%
\mathbb{Z}
_{t_{2}},...,%
\mathbb{Z}
_{t_{n}}$ respectively. For each $i=2,...,n$, set $H_{i}=\langle \sigma
\sigma _{i}^{-1}\rangle $ and $H_{1}=%
\mathbb{Z}
_{t}.$To see that all the required properties are satisfied, first let $%
H=H_{1}\cdot H_{2}\cdots H_{n}$, then $H$ is a subgroup of $G$ containing $%
\sigma $ and $\sigma \sigma _{i}^{-1}$ for all $i,$ hence $H$ contains $%
\sigma $ and $\sigma _{i}$ for all $i,$ therefore $H=G.$ In addition, it is
clear that $H_{1}\cdot D=G,$ and for each $i\geqslant 2,$ $H_{i}\cdot D$
contains $\sigma \sigma _{i}^{-1}$ and $\sigma _{j}$ for all $j\geqslant 2,$
thus $H_{i}\cdot D$ contains $\sigma $ and $\sigma _{j}$ for all $j\geqslant
2,$ therefore $H_{i}\cdot D=G$ for all $i.$
\end{proof}

Dualizing the previous Proposition, we get:

\begin{corollary}
\label{grpth2}Let $G$ a finite abelian group of exponent $t$. Then there
exist an integer $n\geq 1,$ a family $\{H_{i}\}_{i=1}^{n}$ of subgroups of $%
G $ \ with $G/H_{i}$ cyclic of order $t$ for each $i,$ and there exists a
cyclic subgroup $D$ of $G$ of order $t$ satisfying $H_{1}\cap H_{2}\cap
\cdots \cap H_{n}=(1)$ and $H_{i}\cap D=(1)$ for all $i.$
\end{corollary}

\begin{proof}
Since $G\cong \widehat{G}(:=\limfunc{Hom}_{%
\mathbb{Z}
}(G,%
\mathbb{Q}
/%
\mathbb{Z}
))$, it is enough to prove the result for $\widehat{G}.$ A way to justify
the previous isomorphism is to observe that the isomorphism is trivial for
cyclic groups, and appeal to the fundamental theorem of finite abelian
groups together with the fact that $Hom_{%
\mathbb{Z}
}(-,%
\mathbb{Q}
/%
\mathbb{Z}
)$ preserves direct products. Now note that $G$ has subgroups $%
\{H_{i}\}_{i=1}^{n}$ and $D$ as in Proposition \ref{grpth1}. For each $%
i=1,...,n,$ set $K_{i}=H_{i}^{\perp }$ and $D^{\prime }=D^{\perp }$ (recall
that if $H$ is a subgroup of $G$, then $H^{\perp }:=\{f:G\rightarrow 
\mathbb{Q}
/%
\mathbb{Z}
;H\subseteq \ker f\}$ is a subgroup of $\widehat{G}$). On the other hand,
for each $i,$ $\widehat{G}/H_{i}^{\perp }\cong \widehat{H_{i}}\cong H_{i}$,
so for each $i,$ $\widehat{G}/K_{i}$ is cyclic of order $t.$ In addition,
for each $i$, $D^{\perp }\cap H_{i}^{\perp }=(D\cdot H_{i})^{\perp
}=G^{\perp }=(1),$ and $\overset{n}{\underset{i=1}{\dbigcap }}H_{i}^{\perp
}=(H_{1}\cdot H_{2}\cdots H_{n})^{\perp }=G^{\perp }=(1),$ therefore $%
\overset{n}{\underset{i=1}{\dbigcap }}K_{i}=(1)$ and $K_{i}\cap D^{\prime
}=(1)$ for each $i$ as needed.
\end{proof}

Corollary \ref{grpth2} and the Tchebotarev Density Theorem now yield:

\begin{proposition}
\label{tchebo}Let $F$ a global field and $K/F$ a finite abelian Galois
extension with Galois group $G$ of exponent $t$. Then there exists cyclic
subextensions $\{L_{i}/F\}_{i=1}^{n}$ of $K/F$ of degree $t$ so that $%
K=L_{1}...L_{n}$ and there exists infinitely many $P$ of $F$ so that $%
[L_{iP}:F_{P}]=t$ for all $i.$
\end{proposition}

\begin{proof}
$G$ has subgroups $\{H_{i}\}_{i=1}^{n}$ and $D$ as in Corollary \ref{grpth2}%
. For each \ $i$\ let $L_{i}=\mathcal{F}(H_{i}),$ then $\mathcal{G}%
(L_{i}/F)\cong G/H_{i},$ therefore $L_{i}/F$ is cyclic of degree $t.$ On the
other hand, $\mathcal{G}(K/(L_{1}...L_{n}))\subseteq \mathcal{G}%
(K/L_{i})=H_{i}$ for all $i,$ therefore $\mathcal{G}(K/(L_{1}...L_{n}))%
\subseteq \overset{n}{\underset{i=1}{\cap }}H_{i}=(1),$ hence $\mathcal{G}%
(K/(L_{1}...L_{n}))=(1),$ thus $K=L_{1}...L_{n}.$ Let $\sigma $ be a
generator of $D,$ by the Tchebotarev Density Theorem [Pi, p363], there are
infinitely many primes $P$ of $F$ so that $G_{P}^{K/F}=\langle \sigma
\rangle .$ For each of these primes $P,$and each $i,$ $%
G_{P}^{K/L_{i}}=(G_{P}^{K/F})\cap H_{i}=D\cap H_{i}=(1)$ by Corollary \ref%
{grpth2}$,$therefore for each $i,$ $[K_{P}:L_{iP}]=|G_{P}^{K/L_{i}}|=1.$ In
addition, as $\sigma $ has order $t,$ then $%
t=[K_{P}:F_{P}]=[K_{P}:L_{iP}][L_{iP}:F_{P}]=[L_{iP}:F_{P}],$ thus $%
[L_{iP}:F_{P}]=t$ for all $i$ as needed.
\end{proof}

From now\ and for the rest of this section $F$ will be a global field, $K/F$
a finite abelian Galois extension with cyclic subextensions $%
\{L_{i}/F\}_{i=1}^{n}$ as in Proposition \ref{tchebo}.

\begin{remark}
\label{lcm}Let $G$ be a finite abelian group, $\{H_{i}\}_{i=1}^{n}$
subgroups of $G$ with $\overset{n}{\underset{i=1}{\cap }}H_{i}=(1)$ and $%
G/H_{i}$ cyclic for each $i.$ Then $\exp (G)=\func{lcm}\{|G/H_{i}|\}.$ To
see this, note that since $\overset{n}{\underset{i=1}{\cap }}H_{i}=(1),$
then $\widehat{G}=\langle H_{i}^{\perp }\rangle ,$ hence $\exp (G)=\exp (%
\widehat{G})=\func{lcm}\{\exp (H_{i}^{\perp })\}=\func{lcm}\{|G/H_{i}|\}.$
The last equality holds because for each $i$, $G/H_{i}\cong H_{i}^{\perp }$
and $G/H_{i}$ is cyclic.
\end{remark}

In particular if $q$ is a prime of $F$ and $G_{q}=\mathcal{G}(K_{q}/F_{q}),$
let $H_{i}=\mathcal{G}(K_{q}/L_{iq})$ for each $i$, then $G_{q}/H_{i}\cong 
\mathcal{G}(L_{iq}/F_{q})$ and $G_{q}/H_{i}$ is cyclic for each $i.$ As $%
K=L_{1}...L_{n},$ then $K_{q}=L_{1q}...L_{nq},$ hence $\overset{n}{\underset{%
i=1}{\cap }}H_{i}=(1).$ Therefore, by the above $c_{q}:=\exp (G_{q})=\func{%
lcm}\{|G_{q}/H_{i}|\}=\func{lcm}\{[L_{iq}:F_{q}]\}.$ Thus for each prime $q,$
$c_{q}=\func{lcm}\{[L_{iq}:F_{q}]\}.$

\begin{remark}
\label{tilemme copy(1)}If $K=L_{1}...L_{n}$ as in Proposition \ref{tchebo},
then $\func{Dec}(K/F)=\langle \func{Br}(L_{i}/F):i=1,...n\rangle .$ To see
this, note that in [Ti$_{1},$ Lemme 1.3] all that is needed for the argument
is that the $\chi _{i}$ generate the character group $\chi (K/F)$, and this
is equivalent to the $L_{i}$ generating $K$ as fields which is our
assumption.
\end{remark}

\begin{example}
An example of a finite abelian Galois extension with a bad prime. Consider $%
F=%
\mathbb{Q}
(i),$ $K=F(\sqrt{5},\sqrt[4]{14})$ and $q$ an extension of the $5$ adic
valuation $v_{5}$ on $%
\mathbb{Q}
$ to $F.$ We claim that $G\cong 
\mathbb{Z}
_{4}\times 
\mathbb{Z}
_{2}$ and $G_{q}\cong 
\mathbb{Z}
_{2}\times 
\mathbb{Z}
_{2}$, and this would imply that $q$ is a bad prime of $F$ ($c_{q}=2<4=d_{q}$%
). Using the fact that $7$ is prime in $%
\mathbb{Z}
\lbrack i],$ and the Eisenstein Criterion, one sees that $x^{4}-14$ is
irreducible over $F,$ so [$F(\sqrt[4]{14}):F]=4$ $.$ It is also easy to see
that $\sqrt{5}\notin F(\sqrt[4]{14})$ (this is true just because $\sqrt{5}%
\notin 
\mathbb{Q}
(\sqrt[4]{14})$ )$,$ thus $|G|=[K:F]=8$, in addition since $K$ is the
composite of $F(\sqrt[4]{14})$ and $F(\sqrt{5})$ which are linearly disjoint
over $F,$ then $G:=\mathcal{G}(K/F)=\mathcal{G}(F(\sqrt[4]{14})/F)\oplus 
\mathcal{G}(F(\sqrt{5})/F),$ hence $G\cong 
\mathbb{Z}
_{4}\times 
\mathbb{Z}
_{2}.$ It remains to see that $G_{q}\cong 
\mathbb{Z}
_{2}\times 
\mathbb{Z}
_{2}.$ First we have $K_{q}=F_{q}(\sqrt{5},\sqrt[4]{14}),$ and we assert
that $F_{q}(\sqrt[4]{14})=F_{q}(\sqrt{2}).$ To see the last assertion, note
first that as $v_{5}$ has two extensions to $F$ ($5\equiv 1(\func{mod}4)$),
then $\overline{F}=%
\mathbb{Z}
_{5};$ in addition let $L=F(\sqrt[4]{14}),$ then $\overline{L}=\overline{F}(%
\sqrt[4]{\overline{14}})=%
\mathbb{Z}
_{5}(\sqrt[4]{\overline{2}^{2}})=%
\mathbb{Z}
_{5}(\sqrt{\overline{2}}),$ hence as $L_{q}/F_{q}$ is unramified, $%
[L_{q}:F_{q}]=[\overline{L_{q}}:\overline{F_{q}}]=[\overline{L}:\overline{F}%
]=[%
\mathbb{Z}
_{5}(\sqrt{\overline{2}}):%
\mathbb{Z}
_{5}]=2.$ On the other hand, as $\overline{L_{q}}=\overline{L}=%
\mathbb{Z}
_{5}(\sqrt{\overline{2}}),$ then $x^{2}-2$ has a root in $\overline{L_{q}}$,
therefore by the Hensel Lemma, $\sqrt{2}\in L_{q}=F_{q}(\sqrt[4]{14})$,
hence $F_{q}(\sqrt{2})\subseteq F_{q}(\sqrt[4]{14})$ and by degree count $%
F_{q}(\sqrt{2})=F_{q}(\sqrt[4]{14})$ as claimed. Finally we have $%
K_{q}=F_{q}(\sqrt{5},\sqrt[4]{14})=F_{q}(\sqrt{5},\sqrt{2})$. We show that [$%
K_{q}:F_{q}]=4.$ Note that $\sqrt{\overline{2}}\in \overline{K_{q}},$ so [$%
\overline{K_{q}}:\overline{F_{q}}$] $\geqslant $[$\overline{F_{q}}(\sqrt{%
\overline{2}}):F_{P}$]=[$%
\mathbb{Z}
_{5}(\sqrt{\overline{2}}):%
\mathbb{Z}
_{5}$]$=2.$ In addition $\frac{1}{2}+%
\mathbb{Z}
\in \Gamma _{K_{q}}/\Gamma _{F_{q}}$ (being the coset of the value of $\sqrt{%
5}$), then [$\Gamma _{K_{q}}:\Gamma _{F_{q}}]\geq 2.$ Therefore by the
fundamental inequality, [$K_{q}:F_{q}]\geq $[$\overline{K_{q}}:\overline{%
F_{q}}$][$\Gamma _{K_{q}}:\Gamma _{F_{q}}]\geq 4,$ but since $K_{q}=F_{q}(%
\sqrt{5},\sqrt{2}),$ then [$K_{q}:F_{q}]\leq 4$, thus [$K_{q}:F_{q}]=4$ as
claimed, thus $G_{q}\cong 
\mathbb{Z}
_{2}\times 
\mathbb{Z}
_{2}.$
\end{example}

The following result is key to the proof of the main result.

\begin{proposition}
\label{dec=}Let $D\in \func{Br}(F),$ then $D\in \func{Dec}(K/F)$ if and only
if $D_{P}\in \func{Dec}(K_{P}/F_{P})$ for all primes $P$ of $F$.
\end{proposition}

\begin{proof}
It is obvious that $D\in \func{Dec}(K/F)$ implies $D_{P}\in \func{Dec}%
(K_{P}/F_{P})$ for all primes $P$ of $F$.

Conversely, suppose $D_{P}\in \func{Dec}(K_{P}/F_{P})$ for all primes $P$ of 
$F$. Recall that by Proposition \ref{tchebo}, there exists cyclic
subextensions $\{L_{i}/F\}_{i=1}^{n}$ of $K/F$ of degree $t$ so that $%
K=L_{1}...L_{n}$ and there exists infinitely many $P$ of $F$ so that $%
[L_{iP}:F_{P}]=t$ for all $i.$ Recall also that there is a well-known short
exact sequence [see for example Re, 32.13]: 
\begin{equation*}
0\longrightarrow \func{Br}(F)\overset{\oplus r_{F_{P}/F}}{\longrightarrow }%
\underset{P}{\oplus }\func{Br}(F_{P})\overset{\oplus \func{Inv}_{F_{P}}}{%
\longrightarrow }%
\mathbb{Q}
/%
\mathbb{Z}
\longrightarrow 0\text{ \ \ \ \ \ \ \ \ \ \ \ \ (}S\text{)}
\end{equation*}

Let \{$P_{1},...,P_{m}$\} be the support of $D.$ By assumption, $%
D_{P_{j}}\in \func{Dec}(K_{P_{j}}/F_{P_{j}})$ for all $j=1,...,m.$ Since for
each $j,$ $K_{P_{j}}=L_{1P_{j}}...L_{nP_{j}},$ then by Remark \ref{tilemme}
and [Dr, Th. 1 pp.73], for each $j,$ there exist $b_{1j},...,b_{nj}\in
F_{P_{j}}^{\ast }$ such that 
\begin{equation*}
D_{P_{j}}\sim \overset{n}{\underset{i=1}{\otimes }}(L_{iP_{j}}/F_{P_{j}},%
\sigma _{ij},b_{ij}).\text{ \ \ \ \ \ \ \ \ \ \ \ \ \ \ \ \ \ \ \ \ \ \ \ \
\ \ \ \ \ \ \ \ \ \ \ \ \ \ \ \ \ \ (}1\text{)}
\end{equation*}%
For each $i=1,...,n$ and $j=1,...,m,$ write $\func{Inv}%
_{P_{j}}[(L_{iP_{j}}/F_{P_{j}},\sigma _{ij},b_{ij})]=\frac{t_{ij}}{s_{ij}}+%
\mathbb{Z}
$ with $\gcd (t_{ij},s_{ij})=1$. Then $\exp [(L_{iP_{j}}/F_{P_{j}},\sigma
_{ij},b_{ij})]=s_{ij}$ for all $i,j.$ But as $[L_{iP_{j}}:F_{P_{j}}]$
divides $t$ for all $i,j,$ then $(L_{iP_{j}}/F_{P_{j}},\sigma
_{ij},b_{ij})^{\otimes t}\sim 1$ for all $i,j;$ thus $s_{ij}$ divides $t$
for all $i,j.$ Therefore, for each $i=1,...,n,$ there exists $0\leq t_{i}<t$
so that $\dsum\limits_{j=1}^{m}\func{Inv}_{P_{j}}[(L_{iP_{j}}/F_{P_{j}},%
\sigma _{ij},b_{ij})]=\frac{t_{i}}{t}+%
\mathbb{Z}
.$ For each $i=1,...,n$, set $r_{i}=t-t_{i}.$ Then by Proposition \ref%
{tchebo} there exists a non-Archimedean prime $P_{\ast }\notin
\{P_{1},...,P_{m}\}$ so that $[L_{iP_{\ast }}:F_{P_{\ast }}]=t$ for each $%
i=1,...,n.$ Then by Lemma \ref{arch}, there exists $b_{i}^{\ast }\in
F_{P_{\ast }}$ so that $\func{Inv}_{P_{\ast }}[(L_{iP_{\ast }}/F_{P_{\ast
}},\sigma _{i},b_{i}^{\ast })]=\frac{r_{i}}{t}+%
\mathbb{Z}
.$ Therefore, for each $i=1,...,n,$ the element of $\underset{P}{\oplus }%
\func{Br}(F_{P})$ whose $P_{j}^{{}}$ coordinate is $(L_{iP_{j}}/F_{P_{j}},%
\sigma _{ij},b_{ij})$ ($j=1,...m$), the $P_{\ast }^{\prime }s$ coordinate is 
$(L_{iP_{\ast }}/F_{P_{\ast }},\sigma _{i},b_{i}^{\ast }),$ and the $%
P^{\prime }s$ coordinate is $F_{P}$ for $P\notin \{P_{1},...,P_{m},P_{\ast
}\}$ is in the kernel of the map $\oplus \func{Inv}_{F_{P}}.$ Hence by
exactness of ($S$), there exists $A_{i}\in \func{Br}(F)$ so that :

\begin{eqnarray*}
A_{iP_{j}} &\sim &(L_{iP_{j}}/F_{P_{j}},\sigma _{ij},b_{ij})\text{ \ \ \ \ \ 
}j=1,...,m \\
A_{iP_{\ast }} &\sim &(L_{iP_{\ast }}/F_{P_{\ast }},\sigma _{i},b_{i}^{\ast
})\text{ \ \ \ \ \ \ \ \ \ \ \ \ \ \ \ \ \ \ \ \ \ \ \ \ \ \ \ \ \ \ \ \ \ \
\ \ \ \ \ \ \ \ \ \ \ \ \ \ \ \ \ \ (}2\text{)} \\
A_{iP} &\sim &1\text{ \ \ \ \ \ \ \ \ \ \ \ \ \ \ \ \ \ \ \ \ \ \ \ \ \ \ \
for \ }P\notin \{P_{1},...,P_{m},P_{\ast }\}
\end{eqnarray*}

We will complete the proof by the following two claims:

\textbf{Claim 1}: $D\sim \underset{i=1}{\overset{n}{\otimes }}A_{i}$

We need to see that $D_{P}\sim \underset{i=1}{\overset{n}{\otimes }}A_{iP}$
for all primes $P.$ This holds trivially for primes $P\notin
\{P_{1},...,P_{m},P_{\ast }\}$ and easily for primes $P\in
\{P_{1},...,P_{m}\}$ using the relations ($1$) and ($2$). On the other hand, 
$($ $\underset{i=1}{\overset{n}{\otimes }}A_{i})\otimes D^{op}$ has
invariants $%
\mathbb{Z}
$ everywhere except possibly at $P_{\ast },$ but as $($ $\underset{i=1}{%
\overset{n}{\otimes }}A_{i})\otimes D^{op}\in \func{Br}(F),$ the exactness
of the sequence ($S$) implies that the sum of all the invariants of $($ $%
\underset{i=1}{\overset{n}{\otimes }}A_{i})\otimes D^{op}$ is $%
\mathbb{Z}
,$ hence its invariant at $P_{\ast }$ must also be $%
\mathbb{Z}
.$ Therefore $($ $\underset{i=1}{\overset{n}{\otimes }}A_{i})\otimes
D^{op}\sim 1,$ and Claim 1 follows.

\textbf{Claim 2}: $A_{i}\in \func{Br}(L_{i}/F)$ for each $i=1,...,n.$

Again, we just need to prove this locally at each prime of $F$. Let $%
i=1,...,n$ and $P$ a prime of $F;$ if $P=P_{j}\in \{P_{1},...,P_{m}\},$ then 
$A_{iP_{j}}\sim (L_{iP_{j}/F_{P_{j}}},\sigma _{ij},b_{ij})$, therefore $%
A_{iP_{j}}\in \func{Br}(L_{iP_{j}}/F_{P_{j}}).$ If $P=P_{\ast },$ then $%
A_{iP_{\ast }}\sim (L_{iP_{\ast }}/F_{P_{\ast }},\sigma _{i},b_{i}^{\ast }),$
so $A_{iP_{\ast }}\in $\ \ $\func{Br}(L_{iP_{\ast }}/F_{P_{\ast }}).$ If $%
P\notin \{P_{1},...,P_{m},P_{\ast }\},$ then $A_{iP}\sim 1$ and the result
holds trivially.
\end{proof}

\begin{remark}
\label{img}Let $D\in \func{Br}_{t}(K/F),$ then for any prime $q$ of $F,$ $%
D_{q}\in \func{Br}_{t}(K_{q}/F_{q}).$ Write $\func{Inv}_{q}(D_{q})=\frac{r}{s%
}+%
\mathbb{Z}
$ with $\gcd (r,s)=1,$ then $\func{Ind}(D_{q})=\exp (D_{q})=s,$ therefore $s$
divides $t$ and $[K_{q}:F_{q}],$ hence $s$ divides $d_{q}.$ Therefore $D\in 
\func{Br}_{t}(K/F)$ implies for each prime $q$ of $F$, $\func{Inv}%
_{q}(D_{q})=\frac{r}{s}+%
\mathbb{Z}
=\frac{r^{\prime }}{d_{q}}+%
\mathbb{Z}
\in \frac{1}{d_{q}}%
\mathbb{Z}
/%
\mathbb{Z}
.$ On the other hand if $\ D\in \func{Dec}(K/F),$ then $D_{q}\in \func{Dec}%
(K_{q}/F_{q})$ for all primes $q$ of $F.$ But since $K_{q}=L_{1q}\cdots
L_{nq},$ then by Remark \ref{tilemme} $D_{q}\sim A_{1}\otimes \cdots \otimes
A_{n}$ with $A_{i}\in \func{Br}(L_{iq}/F_{q})$ for each $i.$ Hence for each $%
i,$ $\exp (A_{i})=\func{Ind}(A_{i})$ divides $[L_{iq}:F_{q}],$ thus for each
prime $q,$ $\exp (D_{q})$ divides $\func{lcm}(\exp (A_{i}))$ which divides $%
\func{lcm}[L_{iq}:F_{q}]=c_{q}$ [Remark \ref{lcm}]. But as $\func{Inv}%
_{q}(D_{q})=\frac{n_{q}}{\exp (D_{q})}+%
\mathbb{Z}
$ for some integer $n_{q},$ it follows that for each prime $q,$ $\func{Inv}%
_{q}(D_{q})=\frac{n_{q}^{\prime }}{c_{q}}+%
\mathbb{Z}
$ for some integer $n_{q}^{\prime }.$ Therefore $\ D\in \func{Dec}(K/F)$
implies for each prime $q$ of $F$, $\func{Inv}_{q}(D_{q})\in \frac{1}{c_{q}}%
\mathbb{Z}
/%
\mathbb{Z}
.$
\end{remark}

The following is the main result of this paper.

\begin{theorem}
\label{ses}Let $K/F$ a finite abelian Galois extension of global fields
whose Galois group $G:=\mathcal{G}(K/F)$ has exponent $t$. Then there is a
short exact sequence 
\begin{equation*}
0\longrightarrow \func{Dec}(K/F)\overset{f}{\longrightarrow }\func{Br}%
_{t}(K/F)\overset{g}{\longrightarrow }\underset{\text{bad }q^{\prime }s}{%
\oplus }\frac{c_{q}}{d_{q}}%
\mathbb{Z}
/%
\mathbb{Z}
\longrightarrow 0
\end{equation*}%
where $f$ is the inclusion map [Lemma \ref{decin}].
\end{theorem}

\begin{proof}
By Remark \ref{img}, there is a group homomorphism $\varphi :\func{Br}%
_{t}(K/F)\longrightarrow \underset{\text{bad }q^{\prime }s}{\oplus }\frac{1}{%
d_{q}}%
\mathbb{Z}
/%
\mathbb{Z}
$ defined by $\varphi ([D])=\{\func{Inv}_{q}(D_{q})\}_{\text{bad }q^{\prime
}s}$ and the restriction $\phi $ of $\varphi $ to $\func{Dec}(K/F)$ is a
group homomorphism $\func{Dec}(K/F)\longrightarrow \underset{\text{bad }%
q^{\prime }s}{\oplus }\frac{1}{c_{q}}%
\mathbb{Z}
/%
\mathbb{Z}
.$ Let $\mathcal{K=}\ker \varphi $. We assert that $\mathcal{K}\subseteq 
\func{Dec}(K/F).$ To see this, let $D\in \mathcal{K},$ then $D\in \func{Br}%
_{t}(K/F)$ and $\func{Inv}_{q}(D_{q})=%
\mathbb{Z}
$ for each bad prime $q.$ So for each bad prime $q$, $D_{q}\sim 1.$ But on
the other hand since $D\in \func{Br}_{t}(K/F),$ then for each prime $q,$ $%
D_{q}\in \func{Br}_{t}(K_{q}/F_{q}).$ Two cases can be considered:

\textbf{Case 1}: If $q$ is a bad prime, then $D_{q}\sim 1\in \func{Dec}%
(K_{q}/F_{q}).$

\textbf{Case 2}: If $q$ is not a bad prime, then as $D_{q}\in \func{Br}%
_{t}(K_{q}/F_{q}),$ then $\exp (D_{q})=\func{Ind}(D_{q})$ divides $d_{q},$
hence $D_{q}\in \func{Br}_{d_{q}}(K_{q}/F_{q})=\func{Br}%
_{c_{q}}(K_{q}/F_{q})=\func{Dec}(K_{q}/F_{q}).$ The last equality is
Proposition \ref{decloc}, therefore $D_{q}\in \func{Dec}(K_{q}/F_{q})$ for
all primes $q$ of $F,$ hence by Proposition \ref{dec=}, $D\in \func{Dec}%
(K/F) $. We obtain the following commutative diagram where the vertical
arrows are inclusion maps.

\begin{equation*}
\begin{array}{ccccccc}
0\rightarrow & \mathcal{K} & \longrightarrow & \func{Dec}(K/F) & \overset{%
\phi }{\longrightarrow } & \underset{\text{bad }q^{\prime }s}{\oplus }\frac{1%
}{c_{q}}%
\mathbb{Z}
/%
\mathbb{Z}
& \rightarrow 0 \\ 
& {\Large ||} &  & {\LARGE \downarrow \alpha } &  & \downarrow \beta &  \\ 
0\rightarrow & \mathcal{K} & \longrightarrow & \func{Br}_{t}(K/F) & \overset{%
\varphi }{\longrightarrow } & \underset{\text{bad }q^{\prime }s}{\oplus }%
\frac{1}{d_{q}}%
\mathbb{Z}
/%
\mathbb{Z}
& \rightarrow 0%
\end{array}%
\end{equation*}

We show that the rows of the diagram are short exact sequences. But the only
thing to show is that the maps $\varphi $ and $\phi $ are surjective. Let $%
(r_{q})_{\text{bad }q^{\prime }s}\in \underset{\text{bad }q^{\prime }s}{%
\oplus }\frac{1}{d_{q}}%
\mathbb{Z}
/%
\mathbb{Z}
,$ then for each bad prime $q,$ $r_{q}=\frac{t_{q}}{d_{q}}+%
\mathbb{Z}
$ for some integer $t_{q.}$ Recall that by Remark \ref{finbad}, there are
only finitely many bad primes, let $d=\func{lcm}\{d_{q}:q$ bad prime$\}.$ On
the other hand, by Proposition \ref{tchebo}, there exists a prime $P$ of $F$
that is not bad so that $[L_{iP}:F_{P}]=t$ for all $i=1,...,n.$ In addition,
there exists, $D\in \func{Br}(F)$ with:

\begin{eqnarray*}
\func{Inv}_{q}(D_{q}) &=&\frac{t_{q}}{d_{q}}+%
\mathbb{Z}
\text{ \ \ \ \ for each bad prime }q \\
\func{Inv}_{P}(D_{P}) &=&\frac{n^{\prime }}{d}+%
\mathbb{Z}
\\
\func{Inv}_{q}(D_{q}) &=&%
\mathbb{Z}
\text{ \ \ \ \ for }q\notin \{\text{bad primes, }P\}
\end{eqnarray*}

where $n^{\prime }=-d\dsum\limits_{\text{bad }q^{\prime }s}\frac{t_{q}}{d_{q}%
}.$

\textbf{Claim}: $D\in \func{Br}_{t}(K/F)$ and $\varphi ([D])=(r_{q}).$

Note that for each bad prime $q,$ $d_{q}$ divides $t,$ therefore $d$ divides 
$t.$ Also for each prime $q$ (bad or not), $o(\func{Inv}_{q}(D_{q}))$
divides $d$ (where $o(\func{Inv}_{q}(D_{q}))$ denotes the order of $\func{Inv%
}_{q}(D_{q})$ in $%
\mathbb{Q}
/%
\mathbb{Z}
$)$,$ so $\func{lcm}\{o(\func{Inv}_{q}(D_{q})):q\in \{$bad primes, $P\}\}$
divides $d,$ hence as $\exp (D)=\func{lcm}\{o(\func{Inv}_{q}(D_{q})):q\in \{$%
bad primes, $P\}\},$ then $\exp (D)$ divides $d,$ thus $\exp (D)$ divides $t.
$ It remains to show that $D\in \func{Br}(K/F),$ or equivalently $D_{q}\in 
\func{Br}(K_{q}/F_{q})$ for all primes $q$ of $F$ which we now show by
considering the following cases.

\textbf{Case 1}: If $q$ is a bad prime, $\func{Ind}(D_{q})=o(\func{Inv}%
_{q}(D_{q}))$ divides $d_{q}$ and $d_{q}$ divides $|G_{q}|=[K_{q}:F_{q}],$
thus $D_{q}\in \func{Br}(K_{q}/F_{q}).$

\textbf{Case 2}: If $q=P,$ then $\func{Ind}(D_{P})=o(\func{Inv}_{P}(D_{P}))$
divides $d$ and $d$ divides $t=[L_{1P}:F_{P}],$ so $t$ divides $%
[K_{q}:F_{q}],$ thus $\func{Ind}(D_{P})$ divides $[K_{q}:F_{q}],$ hence $%
D_{q}\in \func{Br}(K_{q}/F_{q}).$

\textbf{Case 3}: If $q\notin \{$bad primes, $P\},$ then $D_{q}\sim 1,$ hence 
$D_{q}\in \func{Br}(K_{q}/F_{q}).$

Therefore $D\in \func{Br}_{t}(K/F),$ and clearly by construction of $D$, $%
\varphi ([D])=(r_{q})_{\text{bad }q^{\prime }s}.$

We show that $\phi $ is also surjective. Let $(r_{q})_{\text{bad }q^{\prime
}s}\in \underset{\text{bad }q^{\prime }s}{\oplus }\frac{1}{c_{q}}%
\mathbb{Z}
/%
\mathbb{Z}
.$ Then for each bad prime $q,$ $r_{q}=\frac{t_{q}}{c_{q}}+%
\mathbb{Z}
$ for some integer $t_{q.}$ Recall that by Remark \ref{finbad}, there are
only finitely many bad primes, let $c=\func{lcm}\{c_{q}:q$ bad prime$\}.$
Note that as $c_{q}|d_{q}$ for each $q,$ then $c|d$ and from the above $d|t$%
, it follows that $c|t.$ On the other hand, by Proposition \ref{tchebo},
there exists a prime $P$ of $F$ that is not bad so that $[L_{iP}:F_{P}]=t$
for all $i=1,...,n.$ In addition, there exists, $D\in \func{Br}(F)$ with

\begin{eqnarray*}
\func{Inv}_{q}(D_{q}) &=&\frac{t_{q}}{c_{q}}+%
\mathbb{Z}
\text{ \ \ \ \ for each bad prime }q \\
\func{Inv}_{P}(D_{P}) &=&\frac{m^{\prime }}{c}+%
\mathbb{Z}
\\
\func{Inv}_{q}(D_{q}) &=&%
\mathbb{Z}
\text{ \ \ \ \ for }q\notin \{\text{bad primes, }P\}
\end{eqnarray*}

where $m^{\prime }=-c\dsum\limits_{\text{bad }q^{\prime }s}\frac{t_{q}}{c_{q}%
}.$

\textbf{Claim}: $D\in \func{Dec}(K/F)$ and $\phi ([D])=(r_{q})_{\text{bad }%
q^{\prime }s}.$

To show that $D\in \func{Dec}(K/F)$, it is enough by Proposition \ref{dec=}
to show that $D_{q}\in \func{Dec}(K_{q}/F_{q})$ for all primes $q$ of $F$
which we now show again by considering few cases:

\textbf{Case 1}: If $q$ is a bad prime, $\func{Ind}(D_{q})=o(\func{Inv}%
_{q}(D_{q}))$ divides $c_{q}$ and $c_{q}$ divides $|G_{q}|=[K_{q}:F_{q}],$
thus by [Re, Corollary 31.10], $D_{q}\in \func{Br}_{c_{q}}(K_{q}/F_{q})=%
\func{Dec}(K_{q}/F_{q}).$ The last equality follows from Proposition \ref%
{decloc}.

\textbf{Case 2}: If $q=P,$ then $\func{Ind}(D_{P})=o(\func{Inv}_{P}(D_{P}))$
divides $c$ and $c$ divides $t=[L_{1P}:F_{P}],$ so $t$ divides $%
[K_{q}:F_{q}]=|G_{q}|$ (so $d_{q}=c_{q}=t$)$.$ Thus $\func{Ind}(D_{P})$
divides $[K_{q}:F_{q}],$ hence $D_{q}\in \func{Br}_{c_{q}}(K_{q}/F_{q})=%
\func{Dec}(K_{q}/F_{q})$ again by Proposition \ref{decloc}$.$

\textbf{Case 3}: If $q\notin \{$bad primes, $P\},$ then $D_{q}\sim 1,$ hence 
$D_{q}\in \func{Dec}(K_{q}/F_{q}).$

Therefore $D\in \func{Dec}(K/F),$ and clearly by construction of $D$, $\phi
([D])=(r_{q})_{\text{bad }q^{\prime }s}.$

By the Snake Lemma, 
\begin{equation*}
\func{Br}_{t}(K/F)/\func{Dec}(K/F)\cong (\underset{\text{bad }q^{\prime }s}{%
\oplus }\frac{1}{d_{q}}%
\mathbb{Z}
/%
\mathbb{Z}
)/(\underset{\text{bad }q^{\prime }s}{\oplus }\frac{1}{c_{q}}%
\mathbb{Z}
/%
\mathbb{Z}
)\cong \underset{\text{bad }q^{\prime }s}{\oplus }\frac{c_{q}}{d_{q}}%
\mathbb{Z}
/%
\mathbb{Z}%
\end{equation*}
which completes the proof of the Theorem.
\end{proof}

\begin{corollary}
\label{dec=b}Let $K/F$ be a finite abelian Galois extension of global fields
whose Galois group has a square free exponent $t$ 
\end{corollary}

Then $\func{Dec}(K/F)=\func{Br}_{t}(K/F)$

\begin{proof}
Suppose that $t$ is square free. Let $q$ a prime of $F,$ and prime number $p$
dividing $d_{q},$ then $p$ divides $|G_{q}|,$ thus $G_{q}$ has a subgroup $H$
of order $p,$ thus $p$=$\exp H$ divides $\exp (G_{q})=c_{q}.$ So $c_{q}$ and 
$d_{q}$ have the same prime factors, but as $t$ is square free, so are $c_{q}
$ and $d_{q},$ thus $c_{q}=d_{q}$. Therefore $F$ has no bad primes, and by
the exactness of the sequence of Theorem \ref{ses}, $\func{Dec}(K/F)=\func{Br%
}_{t}(K/F).$
\end{proof}

\section{Division algebras over Henselian valued field with residue a local
or global field}

In this section, we show that if $F$ is a Henselian valued field containing
a primitive $p^{th}$ root of unity so that $\overline{F}$ is either a local
or a global field and $\func{Char}(\overline{F})\neq p$, every prime
exponent division algebra over $F$ is isomorphic to a product of symbol
algebras. We also construct an example to show that the result does not
generalize to higher exponent algebras. In this section, we will use the
terminologies inertial, NSR (nicely semiramified), TR (totally ramified)
division algebras whose definitions and related results can be found in [JW,
sections 1 and 4]. We will also use the terminologies of gauges and
armatures. A good reference for gauges in [TW$_{1}$] and a good reference
for armatures is [TW$_{2}$].

\begin{proposition}
\label{lind} Let $F$ be a Henselian valued field containing a primitive $%
p^{th}$ root of unity such that $\overline{F}$ is either a local or a global
field with $\func{Char}(\overline{F})\neq p$. Then for each $D\in \func{Br}%
_{p}(F),$ $\func{Ind}(D)=p^{l(D)},$ or equivalently any prime exponent
division algebra is isomorphic to a tensor product of symbol algebras of
degree the same prime.
\end{proposition}

\begin{proof}
We will prove that D is isomorphic to a product of symbols. By [Ja, Lemma
3.4], $D$ can be decomposed as $D\sim I\otimes N\otimes T$ where $I$ is
inertial, $N$ is an NSR algebra, $T$ is TR with $\Gamma _{N}\cap \Gamma
_{T}=\Gamma _{F}$ and $\exp (I),\exp (N),\exp (T)$ divides $p.$ Let $S\sim
I\otimes N,$ the underlying division algebra of $I\otimes N.$ Then by [Mo,
Theorem 1], since $\Gamma _{S}\cap \Gamma _{T}=\Gamma _{F},$ $\overline{S}%
\otimes _{\overline{F}}\overline{T}$ is a division algebra and $T$ is
defectless over $F$, then $D\cong S\otimes T.$ Since $T$ is a product of
symbols [Dr$_{2},$ Thm1], it is enough to prove that $S$ is a product of
symbols. As the only prime exponent algebras over number/local fields are
either split or symbols, we get by [JW, Proposition 2.8] $l(I)=l(\overline{I}%
)=0,1.$ If $I\sim 0,$ then $S$\ is NSR and therefore by [JW, Proposition
4.4], is a product of symbols. If $I$ is a symbol, then as $\func{ind}(S)=%
\func{ind}(\overline{I}_{\overline{N}})\func{ind}(N)$ [JW, Theorem 5.15],
then $\func{ind}$($\overline{I}_{\overline{N}}$) divides $p.$ Two cases are
then to be considered. If $\func{ind}(\overline{I}_{\overline{N}})=1$, then $%
\overline{I}\in \func{Br}_{p}(\overline{N}/\overline{F})=\func{Dec}(%
\overline{N}/\overline{F}),$ the last equality is Corollary \ref{dec=b}%
;hence $\overline{I}\in \func{Dec}(\overline{N}/\overline{F}),$ therefore by
[JW, Theorem 5.15], $S$ is NSR. Finally if $\func{ind}(\overline{I}_{%
\overline{N}})=p$, then by the formula $\func{ind}$($S$)=$\func{ind}$($%
\overline{I}_{\overline{N}}$)$\func{ind}$($N$), $[S:F]=[I\otimes N:F]$, thus 
$S\cong I\otimes N,$ therefore $S$ is a product of symbols$.$ Consequently, $%
D$ is isomorphic to a product of symbols, and it follows that $\func{ind}%
(D)=p^{l(D)}.$
\end{proof}

\begin{remark}
The class of fields of Proposition \ref{lind} is the best class for which
the result holds in the following sense. Number fields have transcendence
degree 0 over $%
\mathbb{Q}
,$ and in [ART], there is an example of an indecomposible division algebra
of index 8 over a Henselian valued field $L\ $with $\overline{L}=%
\mathbb{Q}
(t),$ so by [Ti$_{2},$Thm1], its length is 4. Proposition \ref{lind} is
false in general for algebra of non prime exponent as we now show.We start
by the following lemma which computes the value group of a tensor product of
symbol algebras.
\end{remark}

\begin{lemma}
\label{valgroup}Let $A:=\otimes _{k=1}^{m}A_{\omega }(a_{k},b_{k})$ a
division algebra product of symbol algebras of degree $n$ over a Henselian
field $(F,v)$ such that $char(\overline{F})\nmid n.$ If $i_{k},j_{k}$ denote
the standard generators of $A_{\omega }(a_{k},b_{k})$, then $\Gamma
_{A}/\Gamma _{F}=\langle v(i_{k})+\Gamma _{F},v(j_{k})+\Gamma
_{F}:k=1,\ldots ,m\rangle .$
\end{lemma}

\begin{proof}
Consider the canonical armature $\mathcal{A=\langle \{}i_{k}F^{\ast
},j_{k}F^{\ast }:k=1,\ldots ,m\}\mathcal{\rangle \subseteq }A^{\ast
}/F^{\ast }.$ Then the valuation $v$ induces a group homomorphism $%
\widetilde{v}:\mathcal{A\rightarrow }\Gamma _{A}/\Gamma _{F}$ defined by $%
\widetilde{v}(cF^{\ast })=v(c)+\Gamma _{F}.$ For each $a\in \mathcal{A},$
choose in $A$ exactly one representative $x_{a}$ of the coset $a$; then $%
(x_{a})_{a\in \mathcal{A}}$ is a $F$-basis for $A$. Define $y:A\rightarrow
\triangle \cup \{\infty \}$ by $y(\sum_{a\in \mathcal{A}}\lambda _{a}x_{a})=%
\underset{a\in \mathcal{A}}{\min }(v(\lambda _{a}x_{a})).$ Then $y$ is an $F$%
-gauge [TW, Proposition 4.7], and $\Gamma _{A,y}=$ $\func{Im}(\widetilde{v}).
$ But since $y$ is an $F$-gauge, by [TW, Corollary 3.4], $y=v,$ thus $\Gamma
_{A,v}=\func{Im}(\widetilde{v})=\Gamma _{A}/\Gamma _{F}=\langle
v(i_{k})+\Gamma _{F},v(j_{k})+\Gamma _{F}:k=1,\ldots ,m\rangle $ as needed.
\end{proof}

In all that follows, because of using $i$ as a generator of a symbol, we
will denote by $\xi $ the primitive $4^{th}$ root of unity $i$ in $%
\mathbb{C}
$

\begin{proposition}
\label{factgrp}Let $k=%
\mathbb{Q}
(\xi )$ and $L=k(\sqrt{5},\sqrt{13}).$Then the factor group $\func{Br}%
_{4}(L/k){\Large /}\func{Dec}(L/k)$ is a non cyclic abelian group of
exponent 2.
\end{proposition}

\begin{proof}
Note since $L$ is biquadratic, by [Ti$_{2},$Thm1], $\func{Dec}(L/k)=\func{Br}%
_{2}(L/k)$. Therefore it is obvious that $\func{Br}_{4}(L/k){\Large /}\func{%
Dec}(L/k)$ is an Abelian group of exponent 2. To show that $\func{Br}%
_{4}(L/k){\Large /}\func{Dec}(L/k)$ is non cyclic, it is enough to show that
it contains more than two elements which we now proceed to do. More
precisely, we will prove the existence of $A,B\in \func{Br}_{4}(L/k){\Large %
\backslash }\func{Dec}(L/k)$ with $A\nsim B$ $\func{mod}\func{Dec}(L/k).$
Note that since $5\equiv 1(\func{mod}4)$ the prime 5 of $%
\mathbb{Q}
$ splits into two primes $P_{1},P_{2}$ of $k,$ and similarly 13 splits into
two primes $P_{3},P_{4}$ of $k$ so that $[L_{P_{j}}:k_{P_{j}}]=4$ for \ $%
j=1,2,3,4.$ To see that these degrees are all $4$, let's simplify the
notation and denote $P_{1}$ by $P.$ Then $\sqrt{\overline{13}}\in \overline{%
L_{P}},$ so $[\overline{L_{P}}:\overline{k_{P}}]\geqslant \lbrack \overline{%
k_{P}}(\sqrt{\overline{13}}):k_{P}]=[%
\mathbb{Z}
_{5}(\sqrt{\overline{13}}):%
\mathbb{Z}
_{5}]=2.$ In addition $\frac{1}{2}+%
\mathbb{Z}
\in \Gamma _{L_{P}}/\Gamma _{k_{P}}$ (being the coset of the value of $\sqrt{%
5}$), then $[\Gamma _{L_{P}}:\Gamma _{k_{P}}]\geq 2.$ Therefore by the
fundamental inequality, $[L_{P}:k_{P}]\geq \lbrack \overline{L_{P}}:%
\overline{k_{P}}][\Gamma _{L_{P}}:\Gamma _{k_{P}}]\geq 4,$ but since $%
L_{P}=k_{P}(\sqrt{5},\sqrt{13}),$ then$[L_{P}:k_{P}]\leq 4$, thus $%
[L_{P}:k_{P}]=4$ as claimed. The proof for $P_{2}$ is identical to the one
for $P_{1},$ and the proofs for $P_{3}$\ and $P_{4}$ are similar. Now
consider $A,B\in \func{Br}(k)$ with local invariants given in the table below

\begin{tabular}{llllll}
\ \ \ \ \ \ \ \ \ \ \ \ \ \ \  & $P_{1}$ & $P_{2}$ & $P_{3}$ & $P_{4}$ & all
other primes \\ 
$A$ & $\frac{1}{4}+%
\mathbb{Z}
$ & $\frac{1}{4}+%
\mathbb{Z}
$ & $\frac{1}{4}+%
\mathbb{Z}
$ & $\frac{1}{4}+%
\mathbb{Z}
$ & $%
\mathbb{Z}
$ \\ 
$B$ & $\frac{1}{4}+%
\mathbb{Z}
$ & $\frac{1}{4}+%
\mathbb{Z}
$ & $\frac{1}{2}+%
\mathbb{Z}
$ & $%
\mathbb{Z}
$ & $%
\mathbb{Z}
$%
\end{tabular}

Therefore, $\func{ind}(A)=\func{ind}(B)=\exp (A)=\exp (B)=4,$ hence as $%
\func{Dec}(L/k)=\func{Br}_{2}(L/k)$ and $[L_{P_{j}}:k_{P_{j}}]=4$ \ $%
j=1,2,3,4,$ then $A,B\in \func{Br}_{4}(L/k){\Large \backslash }\func{Dec}%
(L/k).$Finally$,$ we assert that $A\nsim B$ $\func{mod}\func{Dec}(L/k)$, for
otherwise there exists $D^{\prime }\in \func{Dec}(L/k)$ with $A\sim B\otimes
D^{\prime }.$ But this would imply that $D^{\prime }$ has local invariants $%
\frac{3}{4}+$ $%
\mathbb{Z}
$ at $P_{3}$, $\frac{1}{4}+%
\mathbb{Z}
$ at $P_{4}$ and $%
\mathbb{Z}
$ at all other primes, which forces $D^{\prime }$ to have index $4.$But
since $\func{ind}(D^{\prime })=$ $\exp (D^{\prime }),$ $D^{\prime }$ would
have exponent 4, which contradicts the fact that $D^{\prime }\in \func{Dec}%
(L/k)=\func{Br}_{2}(L/k)$. Therefore, $A\nsim B$ $\func{mod}\func{Dec}(L/k)$
as claimed.
\end{proof}

\begin{corollary}
\label{extd1}There exists $D_{1}\in \func{Br}_{4}(L/k)$ with $\exp (D_{1})=4$
for which there does not exist $D_{0}\in \func{Dec}(L/k)$ with $D_{1}\sim
A_{\xi }(13,5)\otimes D_{0}.$
\end{corollary}

\begin{proof}
At least one of the algebras $A$ and $B$ of the Proof of Proposition \ref%
{factgrp} will do as they couldn't both be congruent to $A_{\xi }(13,5)$ mod 
$\func{Dec}(L/k)$. \ \ \ \ \ \ \ \ \ \ \ \ \ \ \ \ \ \ \ \ \ \ \ \ \ \ \ \ \ 
\end{proof}

For what follows, $D_{1}$ will denote the algebra from Corollary \ref{extd1}
and $F=k((x))((y))$ the field of Laurent series equipped with the usual
Henselian valuation $v$ so that $\Gamma _{F}=%
\mathbb{Z}
\times 
\mathbb{Z}
$ and $\overline{F}=k.$ Let $I$ the inertial lift of $D_{1}$ over $F;$ and $D
$ the underlying division algebra of $I\otimes _{F}N$ with $N$ denoting the
NSR division algebra $(5,x)\otimes _{F}(13,y).$ Then $\Gamma _{N}=\frac{1}{2}%
\mathbb{Z}
\times \frac{1}{2}%
\mathbb{Z}
$ and $\overline{N}=L.$ Note that by [JW, Proposition 5.15] and the
definition of $D_{1}$ (See Corollary \ref{extd1}). $\overline{I}$ is split
by $\overline{N}=L\cong \overline{D}$ (so $D$ is semiramified), $\Gamma
_{D}=\Gamma _{N}=\frac{1}{2}%
\mathbb{Z}
\times \frac{1}{2}%
\mathbb{Z}
$ and $\func{ind}(D)=\func{ind}(N)=4$. In addition there is a group
isomorphism $I\func{Br}(F)\longrightarrow \func{Br}(k)$ sending $I$ to $D_{1}
$ [JW, Theorem 2.8];  therefore, $\exp (I)=\exp (D_{1})=4.$ Since $\func{ind}%
(D)=4$, then $\exp (D)$ is $2$ or $4$, but since $\exp (I)=4$ and $\exp
(N)=2,$ then $\exp (D)=4.$

\begin{theorem}
\label{counterexple}The division algebra $D$ above has exponent 4 and is not
isomorphic to a product of symbol algebras.
\end{theorem}

\begin{proof}
If $D$ is isomorphic to a product of symbol algebras$,$ then as $\func{Ind}%
(D)=exp(D)=4$, then $D$ is isomorphic to a single symbol division algebra $%
A_{\xi }(t,s)$ of degree $4.$ We let $i,j$ be the standard generators$.$
First, we show that if $D\cong A_{\xi }(t,s)$, then $D\cong A_{\xi
}(ax^{2},by^{2})$ for some $v$-units $a,b$ with $\overline{a},\overline{b}%
\notin \overline{F}^{\ast 2}.$ Note that 
\begin{equation*}
\{(0,0),(0,\frac{1}{2}),(\frac{1}{2},0),(\frac{1}{2},\frac{1}{2})\}\cong (%
\frac{1}{2}%
\mathbb{Z}
\times \frac{1}{2}%
\mathbb{Z}
){\LARGE /}(%
\mathbb{Z}
\times 
\mathbb{Z}
)\cong \Gamma _{D}{\Large /}\Gamma _{F}
\end{equation*}%
Since by Lemma \ref{valgroup} $\Gamma _{D}/\Gamma _{F}=\langle v(i)+\Gamma
_{F},v(j)+\Gamma _{F}\rangle $ and $\Gamma _{D}/\Gamma _{F}$ is not cyclic,
then $v(i)+\Gamma _{F}\neq v(j)+\Gamma _{F}.$ In addition, as $A_{\xi
}(t,s)\cong A_{\xi }(s,t^{-1})\cong A_{\xi }(t,-ts),$ we can assume $%
v(i)+\Gamma _{F}=(\frac{1}{2},0)+\Gamma _{F}$ and $v(j)+\Gamma _{F}=(0,\frac{%
1}{2})+\Gamma _{F}.$ Thus, there is $\gamma \in \Gamma _{F}$ so that $v(i)=(%
\frac{1}{2},0)+\gamma ,$ hence $v(t)=(2,0)+4\gamma =v(x^{2})+4\gamma .$
Therefore, there exists $\delta ,a\in F$ with $a$ unit so that $t=a\delta
^{4}x^{2};$ and similarly, there exists $\epsilon ,b\in F$ with $b$ unit so
that $t=b\epsilon ^{4}y^{2}.$ Thus $D\cong A_{\xi }(t,s)=A_{\xi }(a\delta
^{4}x^{2},b\epsilon ^{4}y^{2})\cong A_{\xi }(ax^{2},by^{2})$ as claimed. In
addition, if either $\overline{a}$ or $\overline{b}$ was a square in $%
\overline{F}^{\ast },$ then by Hensel's lemma, $a$ or $b$ is a square in $F$%
, which implies that $A_{\xi }(ax^{2},bx^{2})$ is similar to a quaternion
algebra and violates the exponent $4$ assumption. Furthermore if $\overline{F%
}(\sqrt{\overline{a}})\mathcal{=}\overline{F}(\sqrt{\overline{b}}),$ then
there exists $\widetilde{r}\in \overline{F},$ so that $\overline{a}=%
\overline{b}\widetilde{r}^{2},$ there exists a $v$-unit $r\in F$ with $a=%
\func{Br}^{2};$ hence $D\cong A_{\xi }(ax^{2},by^{2})=A_{\xi }(\func{Br}%
^{2}x^{2},by^{2})\sim A_{\xi }(b,b)\otimes (b,y)\otimes (rx,b)\sim A_{\xi
}(b,-1)\otimes (b,y)\otimes (rx,b)$ which is exponent 2, therefore $[%
\overline{F}(\sqrt{\overline{a}},\sqrt{\overline{b}}):\overline{F}]=4.$ Next
note that $\sqrt{\overline{a}},\sqrt{\overline{b}}\in \overline{D}$ since $%
\overline{a}=\overline{i^{2}x^{-1}}^{2}$ and $\overline{b}=\overline{%
j^{2}y^{-1}}^{2}$. Then $\overline{F}(\sqrt{\overline{a}},\sqrt{\overline{b}}%
)=\overline{D}$ because $\overline{F}(\sqrt{\overline{a}},\sqrt{\overline{b}}%
)\subseteq \overline{D}$ and $[\overline{D}:\overline{F}]=$ $[\overline{F}(%
\sqrt{\overline{a}},\sqrt{\overline{b}}):\overline{F}]=4.$ Therefore $%
\overline{D}\cong \overline{F}(\sqrt{5},\sqrt{13})\cong \overline{F}(\sqrt{%
\overline{a}},\sqrt{\overline{b}}).$ But as $\theta _{D}=\theta _{N}$ by JW,
Prop. 5.15(a), it follows that $\overline{F}(\sqrt{5})=\mathcal{F(\theta }%
_{N}(\langle (0,\frac{1}{2})+%
\mathbb{Z}
\times 
\mathbb{Z}
\rangle )\mathcal{)=F(\theta }_{D}(\langle (0,\frac{1}{2})+%
\mathbb{Z}
\times 
\mathbb{Z}
\rangle )\mathcal{)=}\overline{F}(\sqrt{\overline{b}}).$ Therefore, $%
\overline{F}(\sqrt{5})\mathcal{=}\overline{F}(\sqrt{\overline{b}}),$ and
similarly $\overline{F}(\sqrt{13})\mathcal{=}\overline{F}(\sqrt{\overline{a}}%
).$ Hence, there exists $\widetilde{\alpha },\widetilde{\beta }\in \overline{%
F}$ so that $\overline{a}=13\widetilde{\alpha }^{2}$ and $\overline{b}=5%
\widetilde{\beta }^{2},$ and by Hensel's Lemma there exists $v$-units $%
\alpha ,\beta \in F$ so that $a=13\alpha ^{2}$ and $b=5\beta ^{2}.$ The
symbols manipulation yields

\begin{eqnarray*}
D &\cong &A_{\xi }(ax^{2},by^{2})=A_{\xi }(13\alpha ^{2}x^{2},5\beta
^{2}y^{2}) \\
&\sim &A_{\xi }(13,5)\otimes A_{\xi }(13,\beta ^{2})\otimes A_{\xi
}(13,x^{2})\otimes A_{\xi }(\alpha ^{2},5)\otimes A_{\xi }(x^{2},5) \\
&\sim &A_{\xi }(13,5)\otimes (5,\alpha )\otimes (13,\beta )\otimes
(5,x)\otimes (13,y) \\
&\sim &A_{\xi }(13,5)\otimes (5,\alpha )\otimes (13,\beta )\otimes N
\end{eqnarray*}

But since $D\sim I\otimes N$, it follows that $I\sim A_{\xi }(13,5)\otimes
(5,\alpha )\otimes (13,\beta )$ and passing to the residue(which can be
justified by JW, Examples 2.4(i), Proposition 2.5 and Theorem 2.8), we
obtain $D_{1}=\overline{I}\sim A_{\xi }(13,5)\otimes (5,\overline{\alpha }%
)\otimes (13,\overline{\beta })$ which contradicts Corollary \ref{extd1}.
\end{proof}

\begin{remark}
In his Ph.D thesis, F. Chang studied tame division algebras over generalized
local fields which are Henselian field with finite residue fields. He was
able to prove that any tame division algebra over such field is isomorphic
to a product of cyclic algebras. More details about these algebras can be
found in[Ch]. Based on that result, we were curious to know whether or not
we could generalized Proposition \ref{lind} to higher exponent division
algebra, but as Theorem \ref{counterexple} shows, this is not the case.
\end{remark}

\end{document}